\documentclass[11pt, reqno]{amsart}
\usepackage[T1]{fontenc}
\usepackage{amsfonts}
\usepackage{amsmath}
\usepackage{amsthm}
\usepackage{amssymb}
\usepackage{geometry}
\usepackage{graphicx}
\usepackage{xcolor}
\usepackage{mathtools}
\usepackage[colorlinks=true, linkcolor=blue]{hyperref}
\usepackage{cleveref}
\usepackage[english]{babel}
\usepackage{lineno}
\usepackage{float}

\newtheorem{theorem}{Theorem}[section]

\newtheorem{corollary}[theorem]{Corollary}
\newtheorem{lemma}[theorem]{Lemma}
\newtheorem{conjecture}[theorem]{Conjecture}
\newtheorem{proposition}[theorem]{Proposition}
\newtheorem{problem}[theorem]{Problem}
\usepackage{tikz}
\usetikzlibrary{positioning}
\usetikzlibrary{decorations,arrows}
\usetikzlibrary{decorations.markings}
\numberwithin{equation}{section}

\newcommand{\N}{\mathbb N}
\newcommand{\Z}{\mathbb Z}
\newcommand{\R}{\mathbb R}
\newcommand{\Q}{\mathbb Q}
\newcommand{\F}{\mathbb F}

\usepackage{rustic}
\usepackage[T1]{fontenc}

\emergencystretch=\maxdimen
\title{New developments on graph sum index}

\author{Dheer Noal Desai}
\address[]{Department of Mathematical Sciences, University of Memphis, Memphis, TN 38152, USA}
\email{dndesai@memphis.edu}

\author{Runze Wang}
\address[]{Department of Mathematical Sciences, University of Memphis, Memphis, TN 38152, USA}
\email{runze.w@hotmail.com}

\thanks{}
\date{\today}

\begin{document}

\sloppy

\begin{abstract}
In a graph, we assign distinct integers to the vertices, and take the sum of two integers if they are on two adjacent vertices. The minimum possible number of different sums is the \emph{sum index} of this graph. In this paper, we present some new developments on graph sum index. First, we explain the connections between graph sum index and results in additive combinatorics. Then, we determine the sum indices of the complete multipartite graphs, hypercubes, and some cluster graphs. Also, we study the maximum number of edges in a graph with a fixed sum index, which is related to the forbidden subgraph problem.

\end{abstract}

\maketitle

\section{Introduction}
Every graph in this paper is assumed to be a finite simple graph.

In a graph, if we assign distinct integers to the vertices, and take the sum of two integers if they are on a pair of adjacent vertices, then at least how many different sums do we get? The minimum possible number of different sums is defined to be the \emph{sum index} of this graph.

Formally, for a graph $G=(V,\ E)$, we can use an injective map $f:V\lhook\joinrel\longrightarrow \Z$ to assign distinct integers to the vertices. For a chosen $f$, a vertex $u\in V$, and an edge $vw\in E$, we call $f(u)$ the \emph{rank} of $u$, and call $f(v)+f(w)$ the \emph{rank sum} of $vw$. Let $\sigma(G,\ f):=\{f(v)+f(w):\ vw\in E\}$ be the set of different rank sums in $G$ under $f$. Then, the sum index of $G$, denoted by $S(G)$, is defined by
\begin{align*}
S(G)=\min_{f:V\lhook\joinrel\rightarrow \Z}\{|\sigma(G,\ f)|\}.
\end{align*}

The concept of sum index was initially proposed by Harrington and Wong in \cite{HW}, and has been studied by multiple scholars recently. In \cite{HW}, Harrington and Wong determined the sum indices of the complete graphs, complete bipartite graphs, and caterpillars. In \cite{HHKRW}, Harrington et al. proved two upper bounds on sum index, and determined the sum indices of cycles, spiders, wheels, and grids. They also worked on the structure of the graphs with a fixed sum index, which will be further studied in this paper. In \cite{Has}, Haslegrave determined the sum indices of prisms, and gave two general lower bounds on sum index, one of which will be stated and used in this paper. In \cite{ZW}, Zhang and Wang determined the sum indices of the necklace graphs and the complements of matchings, cycles and paths.

However, the intriguing connections between sum index and additive combinatorics have not been mentioned in any of the existing papers.

In fact, there are some well-known connections between graph theory and additive combinatorics. For example, we have the graph theoretic proofs of Roth's theorem \cite{Zh} and the Balog-Szemer\'edi-Gowers theorem \cite{BS,Go,TV,Zh}. Here, we show two applications of additive combinatorics in graph sum index.

Harrington and Wong \cite{HW} determined the sum indices of the complete bipartite graphs, with a purely graph theoretic proof.

\begin{theorem}[Harrington and Wong \cite{HW}]\label{bipartite}
    Let $K_{m,\ n}$ be the complete bipartite graph with $m$ and $n$ vertices in each part. Then
    \begin{align*}
        S(K_{m,\ n})=m+n-1.
    \end{align*}
\end{theorem}

In fact, we can directly draw this conclusion using Kneser's theorem \cite{Gr,Kn1,Kn2}, which was proved in 1950s. For an additive abelian group $A$ and nonempty subsets $X,\ Y\subseteq A$, the \emph{sumset} of $X$ and $Y$ is defined to be $X+Y=\{x+y:\ x\in X,\ y\in Y\}$, and the \emph{stabilizer} of $X$ is defined to be $\mathsf{H}(X)=\{a\in A:\ \{a\}+X=X\}$. There are multiple ways to present Kneser's theorem, one of which is as follows.
\begin{theorem}[Kneser \cite{Kn1,Kn2}]
    Let $A$ be an additive abelian group and let $X,\ Y\subseteq A$ be finite nonempty sets. Then
    \begin{align*}
        |X+Y|\ge |X+H|+|Y+H|-|H|,
    \end{align*}
where $H:=\mathsf{H}(X+Y)$.
\end{theorem}
If the additive group $A$ is taken to be $\Z$, then $H$ only consists of the identity, thus $|H|=1$, $|X+H|=|X|$, and $|Y+H|=|Y|$. When calculating the sum index of the complete bipartite graph $K_{m,\ n}$, we take the sum of the ranks of two vertices if and only if they are in different parts. So by Kneser's theorem, we have $S(K_{m,\ n})\ge m+n-1$. Then, for the upper bound, if we assign $X=\{1,\ 2,\ ...,\ m\}$ to the vertices in one part, and assign $Y=\{m+1,\ m+2,\ ...,\ m+n\}$ to the vertices in the other part, then we have $X+Y=\{m+2,\ m+3,\ ...,\ 2m+n\}$, so $|X+Y|=m+n-1$, which means $S(K_{m,\ n})\le m+n-1$; hence $S(K_{m,\ n})=m+n-1$.

With a graph theoretic proof, Harrington and Wong \cite{HW} also determined the sum indices of the complete graphs.

\begin{theorem}[Harrington and Wong \cite{HW}]\label{complete}
    Let $K_n$ be the complete graph with $n$ vertices. Then
    \begin{align*}
        S(K_n)=2n-3.
    \end{align*}
\end{theorem}

This result can also be seen as a corollary of a well-known result in additive combinatorics. For two sets $X$ and $Y$ in the same group/field, the \emph{restricted sumset} of $X$ and $Y$ is defined by $X\hat +Y=\{x+y:\ x\in X,\ y\in Y,\ x\neq y\}$. The following result is known as the Erd\H os-Heilbronn conjecture \cite{EG}, and it was proved by Dias da Silva and Hamidoune \cite{DDSH}. Its generalized form $X\hat + Y$ was proved by Alon, Nathanson, and Ruzsa \cite{ANR}.

\begin{theorem}[Dias da Silva and Hamidoune \cite{DDSH}] \label{restricted}
    Let $\F$ be a field and let $X\subseteq \F$ be a finite nonempty subset. Then
    \begin{align*}
        |X\hat +X|\ge \min\{char(\F),\ 2|X|-3\},
    \end{align*}
    where $char(\F)$ is the characteristic of $\F$, which is defined to be $\infty$ if it does not exist.
\end{theorem}

If we take $\F$ to be $\R$, then $char(\R)=\infty$. So if the vertices in $K_n$ receive distinct real numbers for their ranks, and we take the sum of each pair of ranks, then by Theorem \ref{restricted}, we will get at least $2n-3$ different rank sums. Thus, we have $S(K_n)\ge 2n-3$, because $\Z\subseteq \R$. For the upper bound, we can assign $X=\{1,\ 2,\ ...,\ n\}$ to the $n$ vertices in $K_n$, and get $X\hat +X=\{3,\ 4,\ ...,\ 2n-1\}$, so $|X\hat +X|=2n-3$, which implies $S(K_n)\le 2n-3$. Hence, $S(K_n)=2n-3$.

Also, we would like to mention that, in general, when we take the sum of two numbers if they are on two adjacent vertices, essentially we are making a \emph{partial sumset}, which has a variety of applications. For instance, it is involved in the Balog-Szemer\'edi-Gowers theorem \cite{BS,Go,TV,Zh}. A partial sumset $X+_\Gamma Y:=\{x+y:\ (x,\ y)\in \Gamma\}$, where $\Gamma\subseteq X\times Y$, is a subset of $X+Y$. For a graph $G=(V,\ E)$, a vertex subset $U\subseteq V$, and a rank assignment $f$, we let
\begin{align*}
    f(U)=\{f(u):\ u\in U\}.
\end{align*}
Then if we take both $X$ and $Y$ to be $f(V)$, the set of all the numbers assigned to the vertices, and let $\Gamma$ be the set of all $(f(u),\ f(v))$ with $uv\in E$, then $\sigma(G,\ f)$ is just $f(V)+_\Gamma f(V)$. 

Regarding the connections between graph sum index and additive combinatorics, we will give some further ideas and remarks in the last section.

Now, we introduce some basic bounds on graph sum index.

In a graph, we say that two edges are adjacent if they share a common vertex. Because different vertices receive different ranks, we know that two adjacent edges must have different rank sums. So if we treat each rank sum as a color, then we will get a proper edge coloring of $G$ where adjacent edges get different colors. Thus, the sum index of $G$ is lower bounded by its chromatic index $\chi'(G)$, which is either $\Delta(G)$ or $\Delta(G)+1$ (Vizing \cite{Vi}), where $\Delta(G)$ is the maximum vertex degree of $G$. For the upper bound, it is easy to see that, for $G$ and its subgraph $G'$, we have $S(G')\le S(G)$. For $K_n$, the complete graph on $n$ vertices, we have mentioned in Theorem \ref{complete} that $S(K_n)=2n-3$. Hence, for any graph $G$ of order $n$, we know $S(G)\le 2n-3$.

\begin{proposition}
    Let $G$ be a graph of order $n$. Then
    \begin{align*}
        \Delta(G)\le \chi'(G)\le S(G)\le 2n-3.
    \end{align*}
\end{proposition}

There is another lower bound given by Haslegrave in \cite{Has}. In a graph $G$ of order $n$, we assume that the degrees of the $n$ vertices are $\delta_1(G)\le \delta_2(G)\le \delta_3(G)\le ...\le \delta_n(G)$.

\begin{theorem}[Haslegrave \cite{Has}] \label{lb}
    Let $G$ be a graph of order $n$. Then
    \begin{align*}
        S(G)\ge \max_{1\le k\le n-1} (\delta_k(G)+\delta_{k+1}(G)-k).
    \end{align*}
\end{theorem}

In particular, if $G$ is a $d$-regular graph, then $S(G)\ge 2d-1$.

We organize this paper as follows: In Section 2, we determine the sum indices of some (families of) graphs, including the complete multipartite graphs, hypercubes, and cluster graphs with cluster size four. We devote Section 3 to the maximum number of edges in a graph with a fixed sum index, which is related to the forbidden subgraph problem. Some remarks are given in Section 4.

\section{Sum indices of different graphs}

\subsection{Complete multipartite graphs}

First, let us generalize Theorem \ref{bipartite} to the complete multipartite graphs.
\begin{theorem} \label{multipartite}
    For $k\ge 2$, let $n_1\ge n_2\ge ...\ge n_k$ be positive integers, and let $K_{n_1,\ n_2,\ ...,\ n_k}=(V,\ E)$ be the complete $k$-partite graph with $n_i$ vertices in the $i$-th part. Then
    \begin{align*}
        S(K_{n_1,\ n_2,\ ...,\ n_k})=2N-n_1-n_2-1,
    \end{align*}
    where $N:=\sum_{i=1}^k n_i$.
\end{theorem}

\begin{proof}
    Let us assume $V=\bigsqcup_{i=1}^k V_i$, where $V_i$ consists of the vertices in the $i$-th part.

    For the upper bound, taking the arithmetic progression $1,\ 2,\ ...,\ N$, we assign $1,\ 2,\ ...,\ n_1$ to the vertices in $V_1$, assign $N-n_2+1,\ N-n_2+2,\ ...,\ N$ to the vertices in $V_2$, and arbitrarily assign $n_1+1,\ n_1+2,\ ...,\ N-n_2$ to other vertices. It is easy to check that now we have at most $2N-n_1-n_2-1$ different sums, which are between $n_1+2$ and $2N-n_2$.

    For the lower bound, we let $f:V\lhook\joinrel\longrightarrow \Z$ be arbitrary, and assume $f(u)=\min\{f(v):\ v\in V\}$ for $u\in V_\alpha$. Then
    \begin{align*}
        \{f(u)\}+f(V\setminus V_\alpha)
    \end{align*}
    gives us $|V\setminus V_\alpha|$ different sums, with the largest one being $f(u)+f(w)$, where $w$ is the vertex such that $f(w)=\max\{f(v):\ v\in V\setminus V_\alpha\}$. Then we may assume $w\in V_\beta$ with $\beta\neq \alpha$, and
    \begin{align*}
        \{f(w)\}+f(V\setminus V_\beta)
    \end{align*}
    will give us $|V\setminus V_\beta|$ different sums, with the smallest one being $f(u)+f(w)$. So we have 
    \begin{align*}
        |V\setminus V_\alpha|+|V\setminus V_\beta|-1&=2|V|-|V_\alpha|-|V_\beta|-1 \\
        &\ge 2N-n_1-n_2-1
    \end{align*}
    different sums. 
\end{proof}

Furthermore, we show that if a graph $G$ is constructed by adding more edges into $K_{n_1,\ n_2,\ ...,\ n_k}$ in certain ways, then the sum index of $G$ is still $2N-n_1-n_2-1$.

For any positive integer $n$, we can construct a graph $L_n$ on $n$ vertices, by letting
\begin{align*}
    V(L_n)=\{v_1,\ v_2,\ ...,\ v_n\}
\end{align*}
and 
\begin{align*}
    E(L_n)=\{(v_i,\ v_j):\ i,\ j\in [1,\ n],\ i\neq j,\ i+j\ge n+2\}.
\end{align*}

We can see that, in $L_n$, vertex $v_i$ has $i-1$ or $i-2$ neighbors.

For two graphs $G_1$ and $G_2$, the \emph{join} of them, denoted by $G_1\vee G_2$, is constructed by connecting every vertex in $G_1$ to every vertex in $G_2$.

\begin{theorem}
    For $k\ge 2$, let $n_1\ge n_2\ge ...\ge n_k$ be positive integers. Then
    \begin{align*}
        S(L_{n_1}\vee L_{n_2}\vee K_{n_3}\vee K_{n_4}\vee ...\vee K_{n_k})=2N-n_1-n_2-1,
    \end{align*}
    where $N:=\sum_{i=1}^k n_i$.
\end{theorem}
\begin{proof}
    We have $S(L_{n_1}\vee L_{n_2}\vee K_{n_3}\vee K_{n_4}\vee ...\vee K_{n_k})\ge 2N-n_1-n_2-1$ because $K_{n_1,\ n_2,\ ...,\ n_k}$ is a subgraph of $L_{n_1}\vee L_{n_2}\vee K_{n_3}\vee K_{n_4}\vee ...\vee K_{n_k}$.

    To show that $S(L_{n_1}\vee L_{n_2}\vee K_{n_3}\vee K_{n_4}\vee ...\vee K_{n_k})\le 2N-n_1-n_2-1$, we assume
    \begin{itemize}
        \item $V(L_{n_1})=\{v_1,\ v_2,\ ...,\ v_{n_1}\}$ and $E(L_{n_1})=\{(v_i,\ v_j):\ i,\ j\in [1,\ n_1],\ i+j\ge n_1+2\}$;
        \item $V(L_{n_2})=\{v'_1,\ v'_2,\ ...,\ v'_{n_2}\}$ and $E(L_{n_2})=\{(v'_i,\ v'_j):\ i,\ j\in [1,\ n_2],\ i+j\ge n_2+2\}$.
    \end{itemize}
    We assign $1,\ 2,\ ...,\ n_1$ to $v_1,\ v_2,\ ...,\ v_{n_1}$, assign $N,\ N-1,\ ...,\ N-n_2+1$ to $v'_1,\ v'_2,\ ...,\ v'_{n_2}$, and arbitrarily assign $n_1+1,\ n_1+2,\ ...,\ N-n_2$ to other vertices. As shown in Theorem \ref{multipartite}, those edges in the subgraph $K_{n_1,\ n_2,\ ...,\ n_k}$ have rank sums between $n_1+2$ and $2N-n_2$. We also have that
    \begin{itemize}
        \item If an edge is inside $L_{n_1}$, then by the definition of $L_n$, this edge has rank sum $\ge n_1+2$;
        \item If an edge is inside $L_{n_2}$, then by the definition of $L_n$, this edge has rank sum $\le 2N-n_2$;
        \item If an edge is inside some $K_{n_i}$ with $3\le i\le k$, then the rank sum of this edge is between $2n_1+3$ and $2N-2n_2-1$.
    \end{itemize}
    So $S(L_{n_1}\vee L_{n_2}\vee K_{n_3}\vee K_{n_4}\vee ...\vee K_{n_k})\le 2N-n_1-n_2-1$, and thus we have the equality.
\end{proof}

This also tells us that, in $K_{n_1,\ n_2,\ ...,\ n_k}$, if we add edges to the first part while keeping it a subgraph of $L_{n_1}$, add edges to the second part while keeping it a subgraph of $L_{n_2}$, and add any edges to the third, fourth, ..., $k$-th parts, then the new graph we get also has sum index $2N-n_1-n_2-1$.

\begin{corollary}
    For $k\ge 2$, let $n_1\ge n_2\ge ...\ge n_k$ be positive integers. If for some graph $G$, we have $K_{n_1,\ n_2,\ ...,\ n_k}$ being a subgraph of $G$, and $G$ being a subgraph of $L_{n_1}\vee L_{n_2}\vee K_{n_3}\vee K_{n_4}\vee ...\vee K_{n_k}$, then 
    \begin{align*}
        S(G)=2N-n_1-n_2-1.
    \end{align*}
\end{corollary}

\subsection{Hypercubes}
We show that the sum indices of \emph{hypercubes} attain the lower bound in Theorem \ref{lb}. The $1$-cube $Q_1$ consists of only two vertices and an edge between them, and the $n$-cube $Q_n$ is defined by $Q_{n-1}\square P_2$, which means $Q_n$ can be constructed by copying $Q_{n-1}$ and connecting every vertex with its copy. 

\begin{theorem}
    We have
    \begin{align*}
        S(Q_n)=2n-1.
    \end{align*}
\end{theorem}

\begin{proof}
    It is easy to see that $Q_n$ is an $n$-regular graph, so by Theorem \ref{lb}, we have $S(Q_n)\ge 2n-1$.
    
    We prove $S(Q_n)\le 2n-1$ by induction. We shall assign $1,\ 2,\ 3,\ ...,\ 2^n$ to the vertices in $Q_n$, and show that under our assignment, we have $2n-1$ different sums, which are $2^n+1$ and $2^n+1\pm 2^i$ with $0\le i\le n-2$. Let us denote $\{2^n+1\}\cup\{2^n+1\pm 2^i:\ 0\le i\le n-2\}$ by $\mathcal{A}_n$ for each $n\in \N$.

    First, we can assign $1$ and $2$ to the two vertices in $Q_1$, and we have one sum, which is $3$. Now, we assume that under rank assignment $f_n:V(Q_n)\lhook\joinrel\longrightarrow \Z$, we have $1,\ 2,\ 3,\ ...,\ 2^n$ assigned to the vertices in $Q_n$, such that each edge has one vertex with rank $\le 2^{n-1}$, the other vertex with rank $>2^{n-1}$; and under this assignment, we have $2n-1$ different sums, which are the elements in $\mathcal{A}_n$. We make a copy of $Q_n$, denoted by $Q_n'$, and connect each vertex $v\in Q_n$ with its copy $v'\in Q_n'$, to get a $Q_{n+1}$. Then we construct a new rank assignment $f_{n+1}$ for $Q_{n+1}$ in the following fashion:
    \begin{itemize}
        \item (A) If $v\in Q_n$, and $f_n(v)\le 2^{n-1}$, then we let $f_{n+1}(v)=f_n(v)$.
        \item (B) If $v\in Q_n$, and $f_n(v)>2^{n-1}$, then we let $f_{n+1}(v)=f_n(v)+2^n$.
        \item (C) For $v'\in Q_n'$, we let $f_{n+1}(v')=-f_n(v)+3\cdot 2^{n-1}+1$, where $v\in Q_n$ and $v'$ is the copy of $v$.
    \end{itemize}
    Now, in $Q_{n+1}$, an edge could be in $Q_n$, or in $Q_n'$, or between $Q_n$ and $Q_n'$.
    \begin{itemize}
        \item For $uv\in E(Q_n)$, we know one of $u$ and $v$ falls into (A), the other falls into (B), so 
        \begin{align*}
            f_{n+1}(u)+f_{n+1}(v)&=f_n(u)+f_n(v)+2^n \\
            &\in \{2^{n+1}+1\}\cup \{2^{n+1}+1\pm 2^i:\ 0\le i\le n-2\} \\
            &\subset \{2^{n+1}+1\}\cup \{2^{n+1}+1\pm 2^i:\ 0\le i\le n-1\} \\
            &=\mathcal{A}_{n+1}.
        \end{align*}
        \item For $u'v'\in E(Q_n')$, we know $f_n(u)+f_n(v)\in \{2^n+1\}\cup \{2^n+1\pm 2^i:\ 0\le i\le n-2\}$, so
        \begin{align*}
            f_{n+1}(u')+f_{n+1}(v')&=-(f_n(u)+f_n(v))+3\cdot 2^n+2 \\
            &=2^{n+1}+1+2^n+1-(f_n(u)+f_n(v)) \\
            &\in \{2^{n+1}+1\}\cup \{2^{n+1}+1\pm 2^i:\ 0\le i\le n-2\} \\
            &\subset \{2^{n+1}+1\}\cup \{2^{n+1}+1\pm 2^i:\ 0\le i\le n-1\} \\
            &=\mathcal{A}_{n+1}.
        \end{align*}
        \item For $vv'\in E(Q_{n+1})$ with $v\in Q_n$ and $v'\in Q_n'$: 
        
        If $v$ falls into (A), then
        \begin{align*}
            f_{n+1}(v)+f_{n+1}(v')&=3\cdot 2^{n-1}+1 \\
            &= 2^{n+1}+1-2^{n-1},
        \end{align*}
        which gives us the smallest element in $\mathcal{A}_{n+1}$.

        If $v$ falls into (B), then
        \begin{align*}
            f_{n+1}(v)+f_{n+1}(v')&=2^n+3\cdot 2^{n-1}+1 \\
            &= 2^{n+1}+1+2^{n-1},
        \end{align*}
        which gives us the largest element in $\mathcal{A}_{n+1}$.
    \end{itemize}

    It is easy to check that under $f_{n+1}$, each edge has one vertex with rank $\le 2^n$, the other vertex with rank $>2^n$, so the same conclusion holds for $Q_{n+1}$.

    So $S(Q_n)\le 2n-1$, and thus $S(Q_n)=2n-1$.  
\end{proof}

\subsection{Cluster graphs}
The last topic in this section is on the cluster graphs with equal-size clusters.

The cluster graph with $n$ clusters each of size $k$, denoted by $nK_k$, is the disjoint union of $n$ complete graphs each with $k$ vertices. For $k=2$, we can easily see that $S(nK_2)=1$. For $k=3$, the sum indices of $nK_3$ were determined by Harrington et al. in \cite{HHKRW}.

\begin{theorem}[Harrington et al. \cite{HHKRW}]
    We have $S(nK_3)=s$, where $s$ is the positive integer such that
    \begin{align*}
        {s-1 \choose 3}<n\le {s\choose 3}.
    \end{align*}
\end{theorem}

We determine the sum indices of $nK_4$. The idea of the proof comes from Theorem 6 in \cite{KS}, a paper by Kittipassorn and Sumalroj on multithreshold graphs.

\begin{theorem}
    We have $S(nK_4)=s$, where $s$ is the positive integer such that
    \begin{align*}
        {\lfloor\frac{s-1}{2}\rfloor\choose 3}+{\lceil\frac{s-1}{2}\rceil\choose 3}<n\le {\lfloor \frac{s}{2}\rfloor\choose 3}+{\lceil \frac{s}{2}\rceil\choose 3}.
    \end{align*}
\end{theorem}

\begin{proof}
    Assume that $s$ is the positive integer such that ${\lfloor\frac{s-1}{2}\rfloor\choose 3}+{\lceil\frac{s-1}{2}\rceil\choose 3}<n\le {\lfloor \frac{s}{2}\rfloor\choose 3}+{\lceil \frac{s}{2}\rceil\choose 3}$. 
    
    First, we show $S(nK_4)\ge s$. For a rank assignment $f:V\lhook\joinrel\longrightarrow \Z$, assume we have $t$ different edge rank sums under $f$, which are $k_1<k_2<...<k_t$. Let $L=\{k_1,\ k_2,\ ...,\ k_{\lfloor \frac{t}{2}\rfloor}\}$ and $U=\{k_{\lfloor \frac{t}{2}\rfloor+1},\ k_{\lfloor \frac{t}{2}\rfloor+2},\ ...,\ k_t\}$. We have $|L|=\bigl\lfloor \frac{t}{2}\bigr\rfloor$ and $|U|=\bigl\lceil \frac{t}{2}\bigr\rceil$.

    We claim that there is a triangle in each $K_4$ with its three edges all having rank sums in $L$ or all having rank sums in $U$. This is because, for a $K_4$ with vertices $v_1$, $v_2$, $v_3$, and $v_4$, we may without loss of generality assume $f(v_1)<f(v_2)<f(v_3)<f(v_4)$. Then we have
        \begin{align*}
            f(v_1)+f(v_2)<f(v_1)+f(v_3)<f(v_2)+f(v_3)<f(v_2)+f(v_4)<f(v_3)+f(v_4).
        \end{align*}
        If $f(v_2)+f(v_3)$ is in $L$, then the three edges in triangle $v_1 v_2 v_3$ all have rank sums in $L$. If $f(v_2)+f(v_3)$ is in $U$, then the three edges in triangle $v_2 v_3 v_4$ all have rank sums in $U$.

    By this claim, we must have 
    \begin{align*}
        {\lfloor \frac{t}{2}\rfloor\choose 3}+{\lceil \frac{t}{2}\rceil\choose 3}\ge n,
    \end{align*}
    because otherwise we have two triangles with the same set of edge rank sums, which means they have the same set of vertex ranks, contradicting the injectivity of $f$. So $t\ge s$, and because this argument works for any rank assignment $f$, we have $S(nK_4)\ge s$.

    Next, we show $S(nK_4)\le s$ by construction.

    As shown in Figure \ref{K4}, if a $K_4$ has edge rank sums $a_1,\ a_2,\ a_3$ and $b_1,\ b_2,\ b_3$ under rank assignment $f$, then we must have 
    \begin{align*}
        a_1+b_1=a_2+b_2=a_3+b_3=f(v_1)+f(v_2)+f(v_3)+f(v_4).
    \end{align*}

    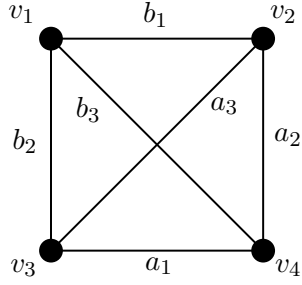
\begin{figure}[H]
        \tikzset{every picture/.style={line width=0.75pt}} 

\begin{tikzpicture}[x=0.75pt,y=0.75pt,yscale=-1,xscale=1]

\draw  [fill={rgb, 255:red, 0; green, 0; blue, 0 }  ,fill opacity=1 ] (284.5,74) .. controls (284.5,70.96) and (286.96,68.5) .. (290,68.5) .. controls (293.04,68.5) and (295.5,70.96) .. (295.5,74) .. controls (295.5,77.04) and (293.04,79.5) .. (290,79.5) .. controls (286.96,79.5) and (284.5,77.04) .. (284.5,74) -- cycle ;
\draw  [fill={rgb, 255:red, 0; green, 0; blue, 0 }  ,fill opacity=1 ] (391.5,74) .. controls (391.5,70.96) and (393.96,68.5) .. (397,68.5) .. controls (400.04,68.5) and (402.5,70.96) .. (402.5,74) .. controls (402.5,77.04) and (400.04,79.5) .. (397,79.5) .. controls (393.96,79.5) and (391.5,77.04) .. (391.5,74) -- cycle ;
\draw  [fill={rgb, 255:red, 0; green, 0; blue, 0 }  ,fill opacity=1 ] (284.5,181) .. controls (284.5,177.96) and (286.96,175.5) .. (290,175.5) .. controls (293.04,175.5) and (295.5,177.96) .. (295.5,181) .. controls (295.5,184.04) and (293.04,186.5) .. (290,186.5) .. controls (286.96,186.5) and (284.5,184.04) .. (284.5,181) -- cycle ;
\draw  [fill={rgb, 255:red, 0; green, 0; blue, 0 }  ,fill opacity=1 ] (391.5,181) .. controls (391.5,177.96) and (393.96,175.5) .. (397,175.5) .. controls (400.04,175.5) and (402.5,177.96) .. (402.5,181) .. controls (402.5,184.04) and (400.04,186.5) .. (397,186.5) .. controls (393.96,186.5) and (391.5,184.04) .. (391.5,181) -- cycle ;
\draw   (290,74) -- (397,74) -- (397,181) -- (290,181) -- cycle ;
\draw    (290,74) -- (397,181) ;
\draw    (290,181) -- (397,74) ;

\draw (336,183) node [anchor=north west][inner sep=0.75pt]   [align=left] {$\displaystyle a_{1}$};
\draw (401,118) node [anchor=north west][inner sep=0.75pt]   [align=left] {$\displaystyle a_{2}$};
\draw (369,102) node [anchor=north west][inner sep=0.75pt]   [align=left] {$\displaystyle a_{3}$};
\draw (335,54) node [anchor=north west][inner sep=0.75pt]   [align=left] {$\displaystyle b_{1}$};
\draw (269,118) node [anchor=north west][inner sep=0.75pt]   [align=left] {$\displaystyle b_{2}$};
\draw (301,102) node [anchor=north west][inner sep=0.75pt]   [align=left] {$\displaystyle b_{3}$};
\draw (267,56) node [anchor=north west][inner sep=0.75pt]   [align=left] {$\displaystyle v_{1}$};
\draw (399,56) node [anchor=north west][inner sep=0.75pt]   [align=left] {$\displaystyle v_{2}$};
\draw (268,184) node [anchor=north west][inner sep=0.75pt]   [align=left] {$\displaystyle v_{3}$};
\draw (401.5,184) node [anchor=north west][inner sep=0.75pt]   [align=left] {$\displaystyle v_{4}$};

\end{tikzpicture}
\caption{Edge rank sums in $K_4$.}
\label{K4}
    \end{figure}

    Thus,
    \begin{align*}
        f(v_1)&=\frac{b_1+b_2-a_3}{2}; \\
        f(v_2)&=\frac{a_2+a_3-a_1}{2}; \\
        f(v_3)&=\frac{a_1+a_3-a_2}{2}; \\
        f(v_4)&=\frac{a_1+a_2-a_3}{2}.
    \end{align*}

    Depending on the parity of $s$, we have two cases.

    \textbf{Case 1.} $s$ is even.

    For $1\le i\le \frac{s}{2}$, let $\alpha_i=\sqrt{p_i}$, where $p_i$ is the $i$-th prime number. Let $\mathsf{p}=\sqrt{p_{\frac{s}{2}+1}}$, the square root of the $(\frac{s}{2}+1)$-th prime number. For $1\le i\le \frac{s}{2}$, let $\beta_i=2\mathsf{p}-\alpha_i$. Let $L=\{\alpha_1,\ \alpha_2,\ ...,\ \alpha_{\frac{s}{2}}\}$ and $U=\{\beta_1,\ \beta_2,\ ...,\ \beta_{\frac{s}{2}}\}$. Assume that the smallest distance between two numbers in $L\cup U$ is $\delta_1>0$.
    
    If we are allowed to assign real numbers to be the ranks the vertices, then we can choose $\{a_1,\ a_2,\ a_3\}$ to be a subset of $L$ or a subset of $U$, and for chosen $\{a_1,\ a_2,\ a_3\}$, we let $b_1=2\mathsf{p}-a_1$, $b_2=2\mathsf{p}-a_2$, and $b_3=2\mathsf{p}-a_3$. So if $\{a_1,\ a_2,\ a_3\}\subseteq L$, then $\{b_1,\ b_2,\ b_3\}\subseteq U$; and if $\{a_1,\ a_2,\ a_3\}\subseteq U$, then $\{b_1,\ b_2,\ b_3\}\subseteq L$. In this way, we can assign edge rank sums for up to ${|L| \choose 3}+{|U| \choose 3}={\frac{s}{2} \choose 3}+{\frac{s}{2} \choose 3}$ $K_4$'s. Then, as explained above, when we know the edge rank sums of the six edges in a $K_4$, we automatically know the vertex ranks of the four vertices. It is easy to see that in a $K_4$, the four vertices have different ranks. By the linear independence of $\{\sqrt{p_i}:\ 1\le i\le \frac{s}{2}+1\}=\{\alpha_1,\ \alpha_2,\ ...,\ \alpha_{\frac{s}{2}},\ \mathsf{p}\}$ over $\Q$, we know that two vertices in different $K_4$'s have different ranks. So now we have an injective rank assignment $V\lhook\joinrel\longrightarrow \R$. Assume that under this rank assignment, the smallest distance between two ranks is $\delta_2>0$.

    However, we need to assign integers instead of real numbers. To do this, we find a multiplier $M$ such that $M\delta_1>10$ and $M\delta_2>10$. For $1\le i\le \frac{s}{2}$, let $\alpha'_i$ be the integer closest to $M\alpha_i$. Then
    \begin{align*}
        |\alpha'_i-M\alpha_i|<1,
    \end{align*}
    so
    \begin{align*}
        d_\alpha:=\max_{1\le i\le \frac{s}{2}}\{|\alpha'_i-M\alpha_i|\}<1.
    \end{align*}
    Let $\mathsf{p}'$ be the integer closest to $M\mathsf{p}$. Then 
    \begin{align*}
        d_\mathsf{p}:=|\mathsf{p}'-M\mathsf{p}|<1.
    \end{align*}
    For $1\le i\le \frac{s}{2}$, let $\beta'_i=2\mathsf{p}'-\alpha'_i$. Then
    \begin{align*}
        |\beta'_i-M\beta_i|&=|(2\mathsf{p}'-\alpha'_i)-M(2\mathsf{p}-\alpha_i)| \\
        &\le 2|\mathsf{p}'-M\mathsf{p}|+|\alpha'_i-M\alpha_i| \\
        &<3,
    \end{align*}
    so
    \begin{align*}
        d_\beta:=\max_{1\le i\le \frac{s}{2}}\{|\beta'_i-M\beta_i|\}<3.
    \end{align*}
    Let $L'=\{\alpha'_1,\ \alpha'_2,\ ...,\ \alpha'_{\frac{s}{2}}\}$ and $U'=\{\beta'_1,\ \beta'_2,\ ...,\ \beta'_{\frac{s}{2}}\}$. Now we know that $L'\cup U'$ is a set of $s$ distinct integers, as the distance between any two integers in $L'\cup U'$ is more than $M\delta_1-2\cdot\max\{d_\alpha,\ d_\beta\}>10-6=4$.

    We choose $\{a_1,\ a_2,\ a_3\}$ to be a subset of $L'$ or a subset of $U'$, and for chosen $\{a_1,\ a_2,\ a_3\}$, we let $b_1=2\mathsf{p}'-a_1$, $b_2=2\mathsf{p}'-a_2$, and $b_3=2\mathsf{p}'-a_3$. So if $\{a_1,\ a_2,\ a_3\}\subseteq L'$, then $\{b_1,\ b_2,\ b_3\}\subseteq U'$; and if $\{a_1,\ a_2,\ a_3\}\subseteq U'$, then $\{b_1,\ b_2,\ b_3\}\subseteq L'$. In this way, we can assign edge rank sums for up to ${\frac{s}{2} \choose 3}+{\frac{s}{2} \choose 3}$ $K_4$'s. As we already explained, these edge rank sums will give us the ranks of the vertices. We know that no two vertices get the same rank, because the distance between any two ranks is more than $M\delta_2-2\cdot \frac{\max\{d_\alpha,\ d_\beta\}+\max\{d_\alpha,\ d_\beta\}+\max\{d_\alpha,\ d_\beta\}}{2}>10-9=1$. So we get an injective rank assignment $V\lhook\joinrel\longrightarrow \Z$.

    \textbf{Case 2.} $s$ is odd.

    For $1\le i\le \lfloor \frac{s}{2} \rfloor$, let $\alpha_i=\sqrt{p_i}$, where $p_i$ is the $i$-th prime number. Let $\mathsf{p}=\sqrt{p_{\lfloor \frac{s}{2} \rfloor+1}}$, the square root of the $(\lfloor \frac{s}{2} \rfloor+1)$-th prime number. For $1\le i\le \lfloor \frac{s}{2} \rfloor$, let $\beta_i=2\mathsf{p}-\alpha_i$. Let $L=\{\alpha_1,\ \alpha_2,\ ...,\ \alpha_{\lfloor \frac{s}{2} \rfloor}\}$ and $U=\{\beta_1,\ \beta_2,\ ...,\ \beta_{\lfloor \frac{s}{2} \rfloor}\}$. Assume that the smallest distance between two numbers in $L\cup U\cup \{\mathsf{p}\}$ is $\delta_1>0$. If we are allowed to assign real numbers to be the ranks the vertices, then, in this case, we have three ways to assign edge rank sums:
    \begin{itemize}
        \item Let $\{a_1,\ a_2,\ a_3\}$ be a subset of $L$, and let $b_1=2\mathsf{p}-a_1$, $b_2=2\mathsf{p}-a_2$, and $b_3=2\mathsf{p}-a_3$, so $\{b_1,\ b_2,\ b_3\}\subseteq U$. In this way, we can assign edge rank sums to ${\lfloor \frac{s}{2}\rfloor \choose 3}$ $K_4$'s.
        \item Let $\{a_1,\ a_2,\ a_3\}$ be a subset of $U$, and let $b_1=2\mathsf{p}-a_1$, $b_2=2\mathsf{p}-a_2$, and $b_3=2\mathsf{p}-a_3$, so $\{b_1,\ b_2,\ b_3\}\subseteq L$. In this way, we can also assign edge rank sums to ${\lfloor \frac{s}{2}\rfloor \choose 3}$ $K_4$'s.
        \item Let $a_1=\mathsf{p}$ and let $\{a_2,\ a_3\}$ be a subset of $L$. Then let $b_1=2\mathsf{p}-a_1=\mathsf{p}$, $b_2=2\mathsf{p}-a_2$, and $b_3=2\mathsf{p}-a_3$. So $a_1=b_1$ and $\{b_2,\ b_3\}\subseteq U$. In this way, we can assign edge rank sums to ${\lfloor \frac{s}{2}\rfloor \choose 2}$ $K_4$'s.
    \end{itemize}
    
    In total, we can assign edge rank sums for up to ${\lfloor \frac{s}{2}\rfloor \choose 3}+{\lfloor \frac{s}{2}\rfloor \choose 3}+{\lfloor \frac{s}{2}\rfloor \choose 2}={\lfloor \frac{s}{2}\rfloor \choose 3}+{\lceil \frac{s}{2}\rceil \choose 3}$ $K_4$'s. When we know the edge rank sums of the six edges in a $K_4$, we automatically know the vertex ranks of the four vertices. It is easy to see that in a $K_4$, the four vertices have different ranks. By the linear independence of $\{\sqrt{p_i}:\ 1\le i\le \lfloor \frac{s}{2} \rfloor+1\}=\{\alpha_1,\ \alpha_2,\ ...,\ \alpha_{\lfloor \frac{s}{2} \rfloor},\ \mathsf{p}\}$ over $\Q$, we know that two vertices in different $K_4$'s have different ranks. So now we have an injective rank assignment $V\lhook\joinrel\longrightarrow \R$. Assume that under this rank assignment, the smallest distance between two ranks is $\delta_2>0$.

    We can find a multiplier $M$ such that $M\delta_1>10$ and $M\delta_2>10$. For $1\le i\le \lfloor \frac{s}{2} \rfloor$, let $\alpha'_i$ be the integer closest to $M\alpha_i$. Then
    \begin{align*}
        |\alpha'_i-M\alpha_i|<1,
    \end{align*}
    so
    \begin{align*}
        d_\alpha:=\max_{1\le i\le \lfloor \frac{s}{2} \rfloor}\{|\alpha'_i-M\alpha_i|\}<1.
    \end{align*}
    Let $\mathsf{p}'$ be the integer closest to $M\mathsf{p}$. Then 
    \begin{align*}
        d_\mathsf{p}:=|\mathsf{p}'-M\mathsf{p}|<1.
    \end{align*}
    For $1\le i\le \lfloor \frac{s}{2} \rfloor$, let $\beta'_i=2\mathsf{p}'-\alpha'_i$. Then 
    \begin{align*}
        |\beta'_i-M\beta_i|&=|(2\mathsf{p}'-\alpha'_i)-M(2\mathsf{p}-\alpha_i)| \\
        &\le 2|\mathsf{p}'-M\mathsf{p}|+|\alpha'_i-M\alpha_i| \\
        &<3,
    \end{align*}
    so
    \begin{align*}
        d_\beta:=\max_{1\le i\le \lfloor \frac{s}{2} \rfloor}\{|\beta'_i-M\beta_i|\}<3.
    \end{align*}
    
    Let $L'=\{\alpha'_1,\ \alpha'_2,\ ...,\ \alpha'_{\lfloor \frac{s}{2} \rfloor}\}$ and $U'=\{\beta'_1,\ \beta'_2,\ ...,\ \beta'_{\lfloor \frac{s}{2} \rfloor}\}$. Now we know that $L'\cup U'\cup \{\mathsf{p}'\}$ is a set of $s$ distinct integers, as the distance between any two integers in $L'\cup U'\cup \{\mathsf{p}'\}$ is more than $M\delta_1-2\cdot\max\{d_\alpha,\ d_\beta,\ d_\mathsf{p}\}>10-6=4$.

    We have three ways to assign edge rank sums:
    \begin{itemize}
        \item Let $\{a_1,\ a_2,\ a_3\}$ be a subset of $L'$, and let $b_1=2\mathsf{p}'-a_1$, $b_2=2\mathsf{p}'-a_2$, and $b_3=2\mathsf{p}'-a_3$, so $\{b_1,\ b_2,\ b_3\}\subseteq U'$. In this way, we can assign edge rank sums to ${\lfloor \frac{s}{2}\rfloor \choose 3}$ $K_4$'s.
        \item Let $\{a_1,\ a_2,\ a_3\}$ be a subset of $U'$, and let $b_1=2\mathsf{p}'-a_1$, $b_2=2\mathsf{p}'-a_2$, and $b_3=2\mathsf{p}'-a_3$, so $\{b_1,\ b_2,\ b_3\}\subseteq L'$. In this way, we can also assign edge rank sums to ${\lfloor \frac{s}{2}\rfloor \choose 3}$ $K_4$'s.
        \item Let $a_1=\mathsf{p}'$ and let $\{a_2,\ a_3\}$ be a subset of $L'$. Then let $b_1=2\mathsf{p}'-a_1=\mathsf{p}'$, $b_2=2\mathsf{p}'-a_2$, and $b_3=2\mathsf{p}'-a_3$. So $a_1=b_1$ and $\{b_2,\ b_3\}\subseteq U'$. In this way, we can assign edge rank sums to ${\lfloor \frac{s}{2}\rfloor \choose 2}$ $K_4$'s.
    \end{itemize}
    In total, we can assign edge rank sums for up to ${\lfloor \frac{s}{2}\rfloor \choose 3}+{\lfloor \frac{s}{2}\rfloor \choose 3}+{\lfloor \frac{s}{2}\rfloor \choose 2}={\lfloor \frac{s}{2}\rfloor \choose 3}+{\lceil \frac{s}{2}\rceil \choose 3}$ $K_4$'s. As we already explained, these edge rank sums will give us the ranks of the vertices. We know that no two vertices get the same rank, because the distance between any two ranks is more than $M\delta_2-2\cdot \frac{\max\{d_\alpha,\ d_\beta,\ d_\mathsf{p}\}+\max\{d_\alpha,\ d_\beta,\ d_\mathsf{p}\}+\max\{d_\alpha,\ d_\beta,\ d_\mathsf{p}\}}{2}>10-9=1$. So we get an injective rank assignment $V\lhook\joinrel\longrightarrow \Z$.
\end{proof}

\section{Maximum number of edges for a fixed sum index}
Given two positive integers $n$ and $N$, at most how many edges can we have in a graph $G$ with $n$ vertices and $S(G)=N$?

The case $N=1$ or $N=2$ was solved in \cite{HHKRW}.

If $N=1$, then it is easy to see that $\Delta(G)\le 1$ and $|E(G)|\le \lfloor \frac{n}{2}\rfloor$, and the equality is attained if and only if $G$ consists of $\lfloor \frac{n}{2}\rfloor$ disjoint copies of $P_2$.

If $N=2$, then $G$ cannot have $K_{1,3}$ or any cycle as a subgraph, because we know $S(K_{1,3})=3$ and $S(C_m)=3$ for any $m$ \cite{HHKRW}. Cycles being forbidden means $G$ is a forest, and $K_{1,3}$ being forbidden further implies $G$ is a linear forest. So $|E(G)|\le n-1$, and the equality is attained if and only if $G$ is a disjoint union of paths.

Now let us handle the case $N=3$.

\begin{theorem} \label{N3}
    Let $G$ be a graph with $n$ vertices and $S(G)=3$. Then
    \[ |E(G)|\le \begin{cases} 
          \frac{3}{2}n-2 & if\ n\ is\ even, \\
          \frac{3}{2}n-\frac{3}{2} & if\ n\ is\ odd,
       \end{cases}
    \]
    and the equality can be attained.
\end{theorem}

\begin{proof}
    First, we note that $S(G)=3$ implies $n\ge 3$.

    It is not possible that $|E(G)|>\frac{3}{2}n$, because otherwise $S(G)\ge \Delta(G)\ge 4$, a contradiction. Also, if $n$ is even, then it is not possible that $|E(G)|=\frac{3}{2}n$. This is because $|E(G)|=\frac{3}{2}n$ means either $S(G)\ge \Delta(G)\ge 4$, a contradiction; or $G$ is a 3-regular graph, which together with Theorem \ref{lb} gives us a contradiction.

    Then for even $n$, we show that $|E(G)|\neq \frac{3}{2}n-1$. If there is a vertex with degree at least four, then $S(G)\ge \Delta(G)\ge 4$, a contradiction. So there are two possible cases left.

    \textbf{Case I.} There is one vertex with degree one, and $n-1$ vertices with degree three.
    
    Let us assume $deg(v)=1$, the three rank sums are $s_1,s_2,s_3$, and the only edge on $v$ has rank sum $s_1$. Then any vertex other than $v$ has an edge on it with rank sum $s_2$, and an edge on it with rank sum $s_3$. So the edges with rank sum $s_2$ form a perfect matching of the induced subgraph on $V(G)\setminus \{v\}$, and it is the same for $s_3$. But the induced subgraph on $V(G)\setminus \{v\}$ has an odd number of vertices, so it does not have a perfect matching, a contradiction.

    \textbf{Case II.} There are two vertices with degree two, and $n-2$ vertices with degree three.

    Let us assume $v_1$ and $v_2$ are the two vertices with degree two, and the rank sums are $s_1,s_2,s_3$.

    \textbf{Subcase II.i.} The rank sums of the two edges on $v_1$ and the rank sums of the two edges on $v_2$ are the same.

    In this subcase, we may assume that under the rank assignment $f:V(G)\lhook\joinrel\longrightarrow \Z$, the two edges on $v_1$ have rank sums $s_1$ and $s_2$, and so do the two edges on $v_2$. Then the edges with rank sum $s_1$ form a perfect matching of $G$, and so do the edges with rank sum $s_2$. But this means $\frac{s_1|V(G)|}{2}=\sum_{v\in V(G)}f(v)=\frac{s_2|V(G)|}{2}$, which implies $s_1=s_2$, a contradiction.

    \textbf{Subcase II.ii.} The rank sums of the two edges on $v_1$ and the rank sums of the two edges on $v_2$ are different.

    In this subcase, we may assume the two edges on $v_1$ have rank sums $s_1$ and $s_2$, and the two edges on $v_2$ have rank sums $s_1$ and $s_3$. Then the edges with rank sum $s_2$ form a perfect matching of the induced subgraph on $V(G)\setminus \{v_2\}$, but this induced subgraph has an odd number of vertices, making it impossible.

    For even $n$, we have proved that $|E(G)|\neq \frac{3}{2}n-1$, so $|E(G)|\le \frac{3}{2}n-2$. The equality can be attained by ladders, as shown in Figure \ref{ladder}, where using the rank assignment marked in blue, we have three different rank sums $-1$, $0$, and $1$. 
    
    \begin{figure}[H]  
        \input ladder.tex
    \end{figure}

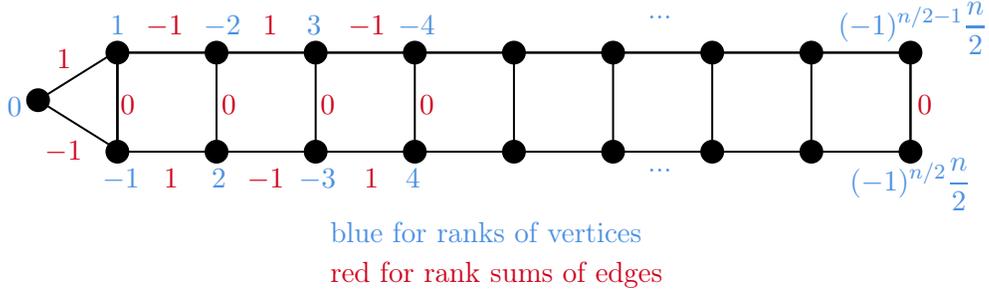
    For odd $n$, we show that $|E(G)|\neq \frac{3}{2}n-\frac{1}{2}$. If there is a vertex with degree at least four, then $S(G)\ge \Delta(G)\ge 4$, a contradiction. Then there must be one vertex with degree two and $n-1$ vertices with degree three. By Theorem \ref{lb}, we know $S(G)\ge \delta_1(G)+\delta_2(G)-1\ge 4$, a contradiction. Thus, for odd $n$, we have $|E(G)|\le \frac{3}{2}n-\frac{3}{2}$. The equality can be attained as shown in Figure \ref{hattedladder}, where we add another vertex to a ladder.

    \begin{figure}[H]     
        \input hattedladder.tex
    \end{figure}
\end{proof}

The rank assignment for ladders in this proof also has applications in other problems. For example, it is used in exactly the same way in \cite{Wa} to minimize the number of edge rank sums.

For the general case, assuming the sum index of a graph $G$ is $N$, we first study the upper bound on $|E(G)|$.

We have an upper bound given by Tur\'an's theorem.

\begin{theorem}[Tur\'an \cite{Tu}] \label{turan}
    Let $G$ be a graph with $n$ vertices. If $G$ does not have $K_{r+1}$ as a subgraph, then
    \begin{align*}
        |E(G)|\le \biggl(1-\frac{1}{r}\biggr)\frac{n^2}{2}.
    \end{align*}
\end{theorem}

It is showed in \cite{HW} that $S(K_{r+1})=2r-1$, so to make sure that $S(G)$ does not exceed $N$, we need to forbid $K_{\lceil \frac{N}{2}\rceil+2}$ to be a subgraph of $G$. So we take $r=\bigl\lceil \frac{N}{2}\bigr\rceil+1$ in Theorem \ref{turan}, and we have $|E(G)|\le (1-\frac{1}{\lceil \frac{N}{2}\rceil+1})\frac{n^2}{2}$.

\begin{proposition} \label{turanbound}
    Let $G$ be a graph with $n$ vertices and $S(G)=N$. Then 
    \begin{align*}
        |E(G)|\le \biggl(1-\frac{1}{\bigl\lceil \frac{N}{2}\bigr\rceil+1}\biggr)\frac{n^2}{2}.
    \end{align*}
\end{proposition}

Using Theorem \ref{lb}, we give another upper bound on $|E(G)|$.

\begin{theorem} \label{ubeg}
    Let $G$ be a graph with $n$ vertices and $S(G)=N$. Then
    \[ |E(G)|\le \begin{cases} 
          \Bigl\lfloor \frac{Nn}{2}-\frac{N^2}{8}-\frac{N}{8}+\frac{\lfloor\frac{N+1}{2}\rfloor}{4}\Bigr\rfloor\approx \frac{Nn}{2}-\frac{N^2}{8} & if\ N\le n-1, \\
          \Bigl\lfloor \frac{Nn}{4}+\frac{n^2}{8}-\frac{N}{4}+\frac{n}{8}+\frac{\lfloor\frac{N+1}{2}\rfloor}{4}-\frac{1}{4}\Bigr\rfloor\approx \frac{Nn}{4}+\frac{n^2}{8}-\frac{N}{8}+\frac{n}{8} & if\ N\ge n.
       \end{cases}
    \]
\end{theorem}

\begin{proof}
    Assume the $n$ vertices have degrees $\delta_1\le \delta_2\le ...\le \delta_n$. We know $\delta_i\le N$ for any $i\in[1,\ n]$ because otherwise $S(G)\ge \Delta(G)>N$, a contradiction.
    
    First assume $N\le n-1$. By Theorem \ref{lb}, we have
    \[\begin{cases}
        \delta_1+\delta_2-1\le N, \\
        \delta_2+\delta_3-2\le N, \\
        ... \\
        \delta_{N-1}+\delta_N-(N-1)\le N.
    \end{cases}
    \]
    We also have $\delta_1\le \lfloor\frac{N+1}{2}\rfloor$ and $\delta_N\le N$. Adding up these inequalities, we have
    \begin{align*}
        2(\delta_1+\delta_2+...+\delta_N)-\frac{N(N-1)}{2}\le N(N-1)+\Bigl\lfloor\frac{N+1}{2}\Bigr\rfloor+N,
    \end{align*}
    which, together with $\delta_{N+1}\le \delta_{N+2}\le ...\le \delta_n\le N$, implies
    \begin{align*}
        |E(G)|=\frac{\delta_1+\delta_2+...+\delta_n}{2}\le \biggl\lfloor \frac{Nn}{2}-\frac{N^2}{8}-\frac{N}{8}+\frac{\lfloor\frac{N+1}{2}\rfloor}{4}\biggr\rfloor.
    \end{align*}

    Next, we assume $N\ge n$. Similarly, we have
    \[\begin{cases}
        \delta_1+\delta_2-1\le N, \\
        \delta_2+\delta_3-2\le N, \\
        ... \\
        \delta_{n-1}+\delta_n-(n-1)\le N.
    \end{cases}
    \]
    We also have $\delta_1\le \lfloor\frac{N+1}{2}\rfloor$ and $\delta_n\le n-1$. Adding up these inequalities, we have
    \begin{align*}
        2(\delta_1+\delta_2+...+\delta_n)-\frac{n(n-1)}{2}\le N(n-1)+\Bigl\lfloor\frac{N+1}{2}\Bigr\rfloor+(n-1),
    \end{align*}
    which implies
    \begin{align*}
        |E(G)|=\frac{\delta_1+\delta_2+...+\delta_n}{2}\le \biggl\lfloor \frac{Nn}{4}+\frac{n^2}{8}-\frac{N}{4}+\frac{n}{8}+\frac{\lfloor\frac{N+1}{2}\rfloor}{4}-\frac{1}{4}\biggr\rfloor.
    \end{align*}
\end{proof}

Given $S(G)=N$, we can see that, if $n$ is small, then the bound given in Proposition \ref{turanbound} by Tur\'an's theorem is better; if $n$ is large, in particular if $n$ goes to infinity, then the bound given in Theorem \ref{ubeg} is better.

Finally, given $S(G)=N$ and assuming the number of vertices in $G$ is not very small, we prove a lower bound on the maximum possible number of edges in $G$.

\begin{theorem} \label{lbeg}
    Let $n$ and $N$ be two positive integers with $2n-3\ge N$. Let $\mathcal{F}_{n,\ N}$ denote the family of graphs with $n$ vertices and sum index $N$. Then 
    \begin{align*}
        \max\{|E(G)|:\ G\in \mathcal{F}_{n,\ N}\}\ge \frac{Nn}{2}-\frac{N^2}{8}-\frac{N}{4}+\epsilon,
    \end{align*}
    with $\epsilon\ge -\frac{1}{8}$ depending on $N$ and $n$.
\end{theorem}

\begin{proof}
    It is sufficient to construct a graph $G$ with $n$ vertices, $\frac{Nn}{2}-\frac{N^2}{8}-\frac{N}{4} + \epsilon$ edges, and $S(G)\le N$. This is because we can then add edges to $G$, and each time we add an edge, $S(G)$ will be increased by at most one, so eventually we will get a graph $G'$ with $n$ vertices, at least $\frac{Nn}{2}-\frac{N^2}{8}-\frac{N}{4} + \epsilon$ edges, and $S(G')=N$.

    We construct such a graph $G$ and a rank assignment in the following way:
    \begin{itemize}
        \item Draw $n$ vertices, and give them ranks $1,\ 2,\ 3,\ ...,\ n$. We will call a vertex $v_i$ if its rank is $i$. In each of the following steps, we will construct some edges with a particular rank sum. The first edge rank sum is taken to be $n+1$, and the rest are taken to be $n+1+1$, $n+1-1$, $n+1+2$, $n+1-2$, $n+1+3$, $n+1-3$ and so on, until we have $N$ edge rank sums in total. We know that $N$ edge rank sums are achievable because $N\le 2n-3$.
        \item Connect $v_1$ with $v_n$; $v_2$ with $v_{n-1}$; $v_3$ with $v_{n-2}$... We get $\bigl\lfloor\frac{n}{2}\bigr\rfloor$ edges each with rank sum $n+1$.
        \item Connect $v_2$ with $v_n$; $v_3$ with $v_{n-1}$; $v_4$ with $v_{n-2}$... We get $\bigl\lfloor\frac{n-1}{2}\bigr\rfloor$ edges each with rank sum $n+2=n+1+1$.
        \item Connect $v_1$ with $v_{n-1}$; $v_2$ with $v_{n-2}$; $v_3$ with $v_{n-3}$... We get $\bigl\lfloor\frac{n-1}{2}\bigr\rfloor$ edges each with rank sum $n=n+1-1$.
        \item ...
    \end{itemize}
    After these steps, we get a graph with $n$ vertices, and the number of edges in this graph is
    \begin{align*}
        \underbrace{\Bigl\lfloor\frac{n}{2}\Bigr\rfloor+\Bigl\lfloor\frac{n-1}{2}\Bigr\rfloor+\Bigl\lfloor\frac{n-1}{2}\Bigr\rfloor+\Bigl\lfloor\frac{n-2}{2}\Bigr\rfloor+\Bigl\lfloor\frac{n-2}{2}\Bigr\rfloor+...+\Bigl\lfloor\frac{n-\lfloor\frac{N}{2}\rfloor}{2}\Bigr\rfloor}_{N\ summands},
    \end{align*}
    which is equal to
    \[ \begin{cases} 
          \frac{Nn}{2}-\frac{N^2}{8}-\frac{N}{4} & if\ N\ is\ even, \\
          \frac{Nn}{2}-\frac{N^2}{8}-\frac{N}{4}+\frac{3}{8} & if\ n\ is\ even\ and\ N\equiv 1\ (mod\ 4),\ or\ n\ is\ odd\ and\ N\equiv 3\ (mod\ 4),\\
          \frac{Nn}{2}-\frac{N^2}{8}-\frac{N}{4}-\frac{1}{8} & if\ n\ is\ even\ and\ N\equiv 3\ (mod\ 4),\ or\ n\ is\ odd\ and\ N\equiv 1\ (mod\ 4).
       \end{cases}
    \]
    
    Under current rank assignment, we have $N$ different edge rank sums in this graph, so its sum index is at most $N$. As explained in the beginning, this suffices to finish the proof. 
\end{proof}

Note that when $N=3$, the construction in this proof is essentially the same as the ladder construction in Theorem \ref{N3}.

Combining Theorem \ref{ubeg} and Theorem \ref{lbeg}, we have that, if we fix $N\ge 4$ and let $\max\{|E(G)|:\ G\in \mathcal{F}_{n,\ N}\}$ be a function about $n$, then 
\begin{align*}
    \max\{|E(G)|:\ G\in \mathcal{F}_{n,\ N}\}+\frac{N^2}{8}\sim \frac{Nn}{2}.
\end{align*}

But there is an extra term $-\frac{N}{4}$ in the lower bound in Theorem \ref{lbeg}, so some new ideas are needed for the exact values of $\max\{|E(G)|:\ G\in \mathcal{F}_{n,\ N}\}$ for $N\ge 4$.

\section{Remarks}

In fact, for a graph $G$, we can generalize the sum index of $G$ to the sum index of $G$ in a general abelian group $A$, by assigning distinct elements in $A$ to the vertices in $G$, and taking the sum of two elements if they are on a pair of adjacent vertices. Under this generalized definition, the results in additive combinatorics will become more powerful tools.

To help readers gain a better understanding, we present a very specific example here. If the abelian group is taken to be $\Z_p^2$ with $p\ge 5$ being a prime number, and the graph $G$ is taken to be $K_{2p+1}$, the complete graph with $2p+1$ vertices, then from a graph theory perspective, it will be hard to give a lower bound on $S_{\Z_p^2}(K_{2p+1})$, the sum index of $K_{2p+1}$ in $\Z_p^2$. However, by the following new result, we know $S_{\Z_p^2}(K_{2p+1})\ge 4p$.

\begin{lemma}[Terkel \cite{Te}]
    Let $p\ge 5$ be a prime number, and let $X\subseteq \Z_p^2$ be a subset with $2p+1$ elements. Then
    \begin{align*}
        |X\hat +X|\ge 4p.
    \end{align*}
\end{lemma}

If we assign the $2p+1$ elements in $\{(0,\ 0)\}\cup \{(1,\ i):\ 0\le i\le p-1\}\cup \{(\frac{p+1}{2},\ j):\ 0\le j\le p-1\}$ to the $2p+1$ vertices, then we will get $4p$ different edge rank sums, which are $(k,\ \ell)$ for $k\in \{1,\ 2,\ \frac{p+1}{2},\ \frac{p+1}{2}+1\}$ and $0\le \ell\le p-1$. So we actually have $S_{\Z_p^2}(K_{2p+1})=4p$.

For future goals, we would like to suggest the following problems.

\begin{problem}
    Determine the exact sum indices of $nK_t$ for $t\ge 5$.
\end{problem}
    
\begin{problem}
    Determine the exact values of $\max\{|E(G)|:\ G\in \mathcal{F}_{n,\ N}\}$ for $N\ge 4$.
\end{problem}

Also, we make the following conjecture.

\begin{conjecture}
    We have $\max\{|E(G)|:\ G\in \mathcal{F}_{n,\ N}\}$ attained by the construction given in Theorem \ref{lbeg}.
\end{conjecture}

\section*{Acknowledgments}
We thank David J. Grynkiewicz for carefully reading the manuscript and giving valuable suggestions.


\begin{thebibliography}{10}

\bibitem{ANR}
N.~Alon, M.~B. Nathanson, and I.~Ruzsa.
\newblock Adding distinct congruence classes modulo a prime.
\newblock {\em Amer. Math. Monthly}, 102(3):250--255, 1995.

\bibitem{BS}
A.~Balog and E.~Szemer\'edi.
\newblock A statistical theorem of set addition.
\newblock {\em Combinatorica}, 14(3):263--268, 1994.

\bibitem{DDSH}
J.~A. Dias~da Silva and Y.~O. Hamidoune.
\newblock Cyclic spaces for {G}rassmann derivatives and additive theory.
\newblock {\em Bull. London Math. Soc.}, 26(2):140--146, 1994.

\bibitem{EG}
P.~Erdős and R.~L. Graham.
\newblock {\em Old and New Problems and Results in Combinatorial Number Theory}.
\newblock L’Enseignement Mathématique, Geneva, 1980.

\bibitem{Go}
W.~T. Gowers.
\newblock A new proof of {S}zemer\'edi's theorem for arithmetic progressions of length four.
\newblock {\em Geom. Funct. Anal.}, 8(3):529--551, 1998.

\bibitem{Gr}
D.~J. Grynkiewicz.
\newblock {\em Structural additive theory}, volume~30 of {\em Developments in Mathematics}.
\newblock Springer, Cham, 2013.

\bibitem{HHKRW}
J.~Harrington, E.~Henninger-Voss, K.~Karhadkar, E.~Robinson, and T.~W.~H. Wong.
\newblock Sum index and difference index of graphs.
\newblock {\em Discrete Appl. Math.}, 325:262--283, 2023.

\bibitem{HW}
J.~Harrington and T.~W.~H. Wong.
\newblock On super totient numbers and super totient labelings of graphs.
\newblock {\em Discrete Math.}, 343(2):111670, 11, 2020.

\bibitem{Has}
J.~Haslegrave.
\newblock Sum index, difference index and exclusive sum number of graphs.
\newblock {\em Graphs Combin.}, 39(2):Paper No. 32, 12, 2023.

\bibitem{KS}
T.~Kittipassorn and T.~Sumalroj.
\newblock Multithreshold multipartite graphs with small parts.
\newblock {\em Discrete Math.}, 347(7):Paper No. 113979, 15, 2024.

\bibitem{Kn1}
M.~Kneser.
\newblock Absch\"atzung der asymptotischen {D}ichte von {S}ummenmengen.
\newblock {\em Math. Z.}, 58:459--484, 1953.

\bibitem{Kn2}
M.~Kneser.
\newblock Ein {S}atz \"uber abelsche {G}ruppen mit {A}nwendungen auf die {G}eometrie der {Z}ahlen.
\newblock {\em Math. Z.}, 61:429--434, 1955.

\bibitem{TV}
T.~Tao and V.~Vu.
\newblock {\em Additive combinatorics}, volume 105 of {\em Cambridge Studies in Advanced Mathematics}.
\newblock Cambridge University Press, Cambridge, 2006.

\bibitem{Te}
J.~Terkel.
\newblock Restricted sums of sets of cardinality $2p+1$ in $\mathbb{Z}_p^2$.
\newblock {\em arXiv preprint arXiv:2410.00143}, 2024.

\bibitem{Tu}
P.~Tur\'an.
\newblock Eine {E}xtremalaufgabe aus der {G}raphentheorie.
\newblock {\em Mat. Fiz. Lapok}, 48:436--452, 1941.

\bibitem{Vi}
V.~G. Vizing.
\newblock On an estimate of the chromatic class of a {$p$}-graph.
\newblock {\em Diskret. Analiz}, (3):25--30, 1964.

\bibitem{Wa}
R.~Wang.
\newblock Threshold numbers of some graphs.
\newblock {\em arXiv preprint arXiv:2406.12063}, 2024.

\bibitem{ZW}
Y.~Zhang and H.~Wang.
\newblock Some new results on sum index and difference index.
\newblock {\em AIMS Math.}, 8(11):26444--26458, 2023.

\bibitem{Zh}
Y.~Zhao.
\newblock {\em Graph theory and additive combinatorics---exploring structure and randomness}.
\newblock Cambridge University Press, Cambridge, 2023.

\end{thebibliography}
\end{document}